\documentclass[a4paper, 11pt]{amsart}

\usepackage[T1]{fontenc}
\usepackage[utf8]{inputenc}

\usepackage{amssymb}

\makeatletter
\@namedef{subjclassname@2020}{%
  \textup{2020} Mathematics Subject Classification}
\makeatother

\allowdisplaybreaks[2]

\newtheorem{theorem}{Theorem}[section]

\newtheorem{lemma}[theorem]{Lemma}
\newtheorem{corollary}[theorem]{Corollary}
\theoremstyle{definition}
\newtheorem{remark}[theorem]{Remark}

\newtheorem{question}[theorem]{Question}
\numberwithin{equation}{section}

\newcommand\V{{\mathcal V}}
\newcommand\A{{\mathcal A}}
\newcommand\U{{\mathcal U}}

\newcommand\C{\mathbb{C}}

\begin{document}

\title[The Waring Problem for Matrix Algebras]{The Waring Problem for Matrix Algebras}


\author{Matej Bre\v sar} 
\address{Faculty of Mathematics and Physics,  University of Ljubljana,  and Faculty of Natural Sciences and Mathematics, University
of Maribor, Slovenia}
\email{matej.bresar@fmf.uni-lj.si}
\author{Peter \v Semrl}
\address{Faculty of Mathematics and Physics,  University of Ljubljana, Slovenia}
\email{peter.semrl@fmf.uni-lj.si}
\thanks{Supported by ARRS Grants P1-0288 and J1-2454.}



\begin{abstract}  If a noncommutative polynomial $f$ is  neither an identity nor a central polynomial of $\A=M_n(\C)$, then every trace zero matrix in $\A$ can be written as a sum of two matrices from $f(\A)-f(\A)$. Moreover, ``two'' cannot be replaced by ``one''.
\end{abstract} 
\maketitle

\newcommand\E{\ell}
\newcommand\mathcalM{{\mathcal M}}
\newcommand\pc{\mathfrak{c}}

\newcommand{\enp}{\begin{flushright} $\Box$ \end{flushright}}

\section{Introduction}

The classical Waring's problem, proposed by Edward  Waring in 1770 and solved by David Hilbert in 1909, asks whether for 
every positive integer $k$ there exists a positive integer $g(k)$ such that every positive integer
can  be expressed as a sum of  $g(k)$ $k$th powers of nonnegative integers.
Various extensions and variations of this problem have been studied by  different groups of mathematicians. One of them is the Waring problem for
finite simple groups, resolved ten years ago by  Larsen, Shalev, and Tiep 
\cite{LST}. They proved that  given a word  $w=w(x_1,\dots,x_m) \ne 1$
and denoting by 
$w(G)$, where $G$ is a group,
  the image of the word map from $G^m$ to $G$
 induced by $w$, 
  $w(G)^2 = G$  for 
every  finite non-abelian simple group $G$ of sufficiently high order.   It is natural to seek analogues of this results for finite-dimensional simple algebras,
with (noncommutative) polynomials  (i.e., elements of free algebras) playing the role of
words (i.e., elements of free groups).

Let us introduce the appropriate setting for our work. Let $F$ be a field,
let $F\langle \mathcal X\rangle$ be the free algebra generated by the set $\mathcal X = \{X_1,X_2,\dots\}$,
 let
$f=f(X_1,\dots,X_m)\in F\langle \mathcal X  \rangle$ be a  polynomial, and let $\A$ be an $F$-algebra. The set
$$f(\A)=\{f(a_1,\dots,a_m)\,|\,a_1,\dots,a_m\in \A\}$$ is called
 the {\em image of $f$} (in $\A$). If $f(\A)=\{0\}$, then we call $f$ a {\em (polynomial) identity} of $\A$. If $f(\A)$ is contained in the center of $\A$ but is not an identity of $\A$,
then  we call $f$  a {\em central polynomial} for $\A$. Further, $f$ is {\em $k$-central} for $\A$, where $k\ge 2$, if $f^k$ is central for $\A$, but $f,\dots,f^{k-1}$ are not.

In the recent paper \cite{B}, the first author proposed the study   
 of the Waring problem for algebras of matrices (not only  over fields, but also over general algebras). One of the  main results, \cite[Theorem 3.18]{B}, states that if
$\A=M_n(F)$, where $n\ge 2$ and 
$F$ is an algebraically closed  field with characteristic $0$,  and $f\in  F\langle \mathcal X\rangle$ is a   polynomial  which is neither an identity nor a central polynomial of $\A$, then
 $$f(\A)-f(\A) = \{s-t\,|\,s,t\in f(\A)\}$$ contains all square-zero matrices in $\A$. Using the fact that every trace zero matrix is a sum of 
four square-zero matrices \cite{dSP}, it follows that every trace zero matrix in $\A$ is  
a sum of four matrices from $f(\A)-f(\A)$ \cite[Corollary 3.19]{B}.  The obvious question that was left open is what is the least number than can replace four. The goal of this paper is to show that, at least  when $F=\C$, the answer is two. More precisely,  the following are our main theorems.

\begin{theorem}\label{main}
Let
$\A=M_n(\C)$, $n\ge 2$, and let 
 $f\in  \C\langle \mathcal X\rangle$ be a polynomial which  is neither an identity nor a central polynomial of $\A$. Then every trace zero matrix in $\A$ can be written as 
a sum of two matrices from $f(\A)-f(\A)$. 
\end{theorem}


\begin{theorem}\label{main2}
Let $F$ be a field of characteristic $0$ and let $\A= M_n(F)$ with  $n > 2$.
If a polynomial $f$ is $2$-central for  $\A$, then 
the set
 of all linear combinations of two matrices from $f(\A)$ does not contain all trace zero matrices.
In particular, not every trace zero matrix lies in $f(\A)-f(\A)$.
\end{theorem}

It should be added here that  $2$-central polynomials for 
$\A=M_n(F)$ indeed exist for some even $n$ \cite[Proposition 2]{BMR}.


Taking a polynomial such as $f=[X_1,X_2]$ we see that the
 restriction to trace zero matrices is indeed 
necessary.
A polynomial  like $f=[X_1,X_2] + \frac{1}{n}$ has only matrices of trace $1$ in its image, which explains why the involvement of the set $f(\A)-f(\A)$ is natural.

The proof of Theorem \ref{main} is based on two results, one known and one new. The first one is \cite[Theorem 3.8]{B} which states that $f(\A)$ contains a matrix $B$ all of whose eigenvalues have  algebraic multiplicity   at most $\frac{n}{2}$. The second one is Theorem \ref{taglavn} which states that if $B\in \A$ is any such matrix, 
then every trace zero matrix can be written as a sum of two matrices of the form $B'-B''$ where $B'$ and $B''$ are similar to $B$.
Section \ref{sp}, which occupies the larger part of the paper, is primarily devoted to proving this theorem.


Theorem \ref{main2} will be proved in Section \ref{lg}.

Theorems \ref{main} and \ref{main2} give a definitive solution of one basic Waring type problem for algebras.  
 We will discuss other problems and establish  some new results in Section \ref{sr}. In particular,
we will show that if $f$ and $\A$ are as in Theorem \ref{main}, then $f(\A)-f(\A)$ already contains all trace zero matrices in $\A$ provided that either $n$ is prime or $f$ is multilinear (Theorem \ref{mp}).  The latter is of interest in light of the L'vov-Kaplansky conjecture. We will also consider
 the Waring problem for 
the algebra $\mathcal B(H)$ of all bounded linear operators on a Hilbert space $H$ (Theorem \ref{bh}).

\section{Proof of Theorem \ref{main}} \label{sp} 

Our goal is to prove Theorem \ref{taglavn}, from which Theorem \ref{main} will  quickly follow. 
To this end, we need a series of lemmas. 

We start with a simple lemma which will not  be needed until the proof of  Theorem \ref{taglavn},  
but may help the reader to understand the motivation for the lemmas that follow.

\begin{lemma}\label{brsvt}
If each eigenvalue of the matrix
 $B \in M_n (\mathbb{C})$  has algebraic multiplicity $\le \frac{n}{2}$, then either
\begin{enumerate}
\item[{\rm (a)}] 
 $n$ is even and $B$ is similar to a block diagonal matrix
$$
\left[ \begin{matrix} B_1 & 0 \cr 0 & B_2 \cr \end{matrix} \right],
$$
where $B_1$ and $B_2$ are $\frac{n}{2} \times \frac{n}{2}$ matrices with disjoint spectra, or
\item[{\rm (b)}] $B$ is similar to a block diagonal matrix
$$
B = \left[ \begin{matrix} B_1 & 0 &0 \cr 0 & B_2 & 0\cr 0 & 0 & B_3 \end{matrix} \right],
$$
where the diagonal blocks are of  sizes $p \times p$, $q\times q$, and $r\times r$ with
$$
p,q,r < {n \over 2}
$$
and the spectra of $B_1$, $B_2$, and $B_3$ are pairwise disjoint.
\end{enumerate}
\end{lemma}

\begin{proof} By the
 Jordan normal form theorem, $B$ is similar to a block diagonal matrix
$$
B = \left[ \begin{matrix} 
C_1 & 0 & 0 & \ldots & 0 \cr
0 & C_2 & 0 & \ldots & 0 \cr
0 & 0 & C_3 & \ldots & 0 \cr
\vdots & \vdots & \vdots & \ddots & \vdots  \cr  0 & 0 & 0 & \ldots & C_k  \cr \end{matrix} \right],
$$
where all matrices $C_1 , \ldots, C_k$ have only one eigenvalue and these eigenvalues are pairwise different.
If $n$ is even  and one of the diagonal blocks is of  size $\frac{n}{2} \times \frac{n}{2}$, then (a) holds. Assume, therefore, that this is not the case. 
We take the maximal integer $j$ such that the size of the matrix
$$
B_1 = \left[ \begin{matrix} 
C_1  & \ldots & 0 \cr
\vdots &  \ddots & \vdots  \cr   0 & \ldots & C_j  \cr \end{matrix} \right]
$$
is at most $\frac{n}{2} \times \frac{n}{2}$. If the size of $B_1$ is $\frac{n}{2} \times \frac{n}{2}$, then by setting
$$
B_2 = \left[ \begin{matrix} 
C_{j+1}  & \ldots & 0 \cr
\vdots &  \ddots & \vdots  \cr   0 & \ldots & C_k  \cr \end{matrix} \right]
$$
we again arrive at (a). Otherwise we take $B_2 = C_{j+1}$ and 
$$
B_3 = \left[ \begin{matrix} 
C_{j+2}  & \ldots & 0 \cr
\vdots &  \ddots & \vdots  \cr   0 & \ldots & C_k  \cr \end{matrix} \right],
$$
so that (b) holds.
\end{proof}

We continue by introducing some notation.
Let $q$ be a real number, $0 < q < 1$. By  $P_2$, $R_q$, and  $R_{q}^-$ we denote  the $2\times 2$ rank one idempotent matrices  
$$
P_2 = \left[ \begin{matrix} 1 & 0 \cr 0 & 0 \cr \end{matrix} \right],  $$
$$ R_q = \left[ \begin{matrix} q & \sqrt{q(1-q)} \cr \sqrt{q(1-q)} & 1-q \cr \end{matrix} \right],$$
and
$$R_{q}^- = \left[ \begin{matrix} q & -\sqrt{q(1-q)} \cr -\sqrt{q(1-q)} & 1-q \cr \end{matrix} \right].
$$
We will need the following simple observations. A $2\times 2$ matrix $A$ commutes with $P_2$ if and only if $A$ is diagonal.  Let $p,q \in (0,1)$, $p\not= q$. If a $2 \times 2$ diagonal matrix $D$ satisfies one of the following four conditions $$R_q D = D R_p,\quad  R_{q}^- D = D R_p,\quad R_q D = D R_{p}^-,\quad R_q^- D = D R_{p}^-,$$
 then $D=0$, and if it satisfies 
$$R_q D = D R_q\quad\mbox{or}\quad R_q^- D = D R_q^-,$$ then $D$ is a scalar multiple of the $2 \times 2$ identity matrix.

The next technical lemma will be needed as a tool for proving Lemmas \ref{jan} and \ref{jan2}.

\begin{lemma}\label{snegec}
Let $n \ge 2k$ be positive integers and let $q_1, \ldots, q_k \in (0,1)$ be pairwise different real numbers. We denote by $P, Q \in M_n (\mathbb{C})$ the block diagonal rank $k$ idempotents
$$
P = \left[ \begin{matrix} 
P_2 & 0 & 0 & \ldots & 0 & 0\cr
0 & P_2 & 0 & \ldots & 0 &0\cr
0 & 0 & P_2 & \ldots & 0 & 0\cr
\vdots & \vdots & \vdots & \ddots & \vdots & \vdots \cr  0 & 0 & 0 & \ldots & P_2 &0 \cr 0 & 0 & 0 & \ldots & 0&0 \cr\end{matrix} \right]\ \ \ {\rm and} \ \ \ 
Q = 
 \left[ \begin{matrix} 
Q_{q_1} & 0 & 0 & \ldots & 0 & 0\cr
0 & Q_{q_2} & 0 & \ldots & 0 & 0\cr
0 & 0 & Q_{q_3} & \ldots & 0 & 0\cr
\vdots & \vdots & \vdots & \ddots & \vdots & \vdots \cr  0 & 0 & 0 & \ldots & Q_{q_k} & 0 \cr 0 & 0 & 0 & \ldots & 0 & 0 \cr\end{matrix} \right],
$$
where $$Q_{q_i}\in \{ R_{q_i},  R_{q_i}^-\},\,\,\, i=1,\dots,k.$$
Assume that $A \in M_n (\mathbb{C})$ commutes with both $P$ and $Q$. Then there exist complex numbers $\lambda_1 , \ldots , \lambda_k$ and an $(n-2k)\times (n-2k)$ matrix $A'$ such that
$$
A = \left[ \begin{matrix} 
\lambda_1 I & 0 & 0 & \ldots & 0 & 0\cr
0 & \lambda_2 I & 0 & \ldots & 0 & 0\cr
0 & 0 & \lambda_3 I & \ldots & 0 & 0\cr
\vdots & \vdots & \vdots & \ddots & \vdots & \vdots \cr  0 & 0 & 0 & \ldots & \lambda_k I & 0 \cr  0 & 0 & 0 & \ldots & 0 & A' \cr \end{matrix} \right].
$$
where $I$ denotes the $2 \times 2$ identity matrix. (If $n=2k$, then the last row and the last column in the above block matrix representations
are of course absent.)
\end{lemma}

\begin{proof}
Let $A \in  M_n (\mathbb{C})$. We  write it in  the block matrix form 
$$
A = \left[ \begin{matrix} 
A_{11} & A_{12} & A_{13} & \ldots & A_{1k} & A_{1, k+1}\cr
A_{21} & A_{22} & A_{23} & \ldots & A_{2k} & A_{2, k+1}\cr A_{31} & A_{32} & A_{33} & \ldots & A_{3k} & A_{3, k+1}\cr
\vdots & \vdots & \vdots & \ddots & \vdots & \vdots\cr
A_{k1} & A_{k2} & A_{k3} & \ldots & A_{kk} & A_{k, k+1}\cr A_{k+1,1} & A_{k+1,2} & A_{k+1,3} & \ldots & A_{k+1,k} & A_{k+1, k+1}\cr
    \end{matrix} \right],
$$
where each of the $A_{ij}$'s, $1 \le i,j \le k$, is a $2\times 2$ matrix. If $A$ commutes with both $P$ and $Q$, then
\begin{equation}\label{dostdost}
P_2 A_{ij} = A_{ij} P_2, \ \ \ i,j = 1, \ldots, k,
\end{equation}
\begin{equation}\label{punkut}
Q_{q_i} A_{ij} = A_{ij} Q_{q_j}, \ \ \ i,j = 1, \ldots, k,
\end{equation}
\begin{equation}\label{punk}
P_2 A_{i, k+1} = Q_{q_i} A_{i, k+1} = 0,  \ \ \ i = 1, \ldots, k,
\end{equation}
and
\begin{equation}\label{punk2}
A_{ k+1, j} P_2  = A_{k+1 , j} Q_{q_j} = 0,  \ \ \ j = 1, \ldots, k.
\end{equation}
Applying the simple observations preceding the formulation of the lemma to  (\ref{dostdost}) and  (\ref{punkut}), we see that the $A_{ii}$'s, $i=1, \ldots, k$, are scalar matrices and $A_{ij} = 0$ whenever $1 \le i,j \le k$ and $i\not=j$. If we consider $P_2 , Q_{q_i} : \mathbb{C}^2 \to \mathbb{C}^2$ as linear operators, then the intersection of their kernels is the zero subspace and the sum of their images is the whole space $\mathbb{C}^2$. Thus,  (\ref{punk}) and  (\ref{punk2}) imply that
$
A_{i, k+1} = 0$, $i = 1, \dots, k$, and $A_{ k+1, j} =0$, $j=1,\dots,k$.
\end{proof}

\begin{remark}\label{zgoljdodatnarazlags}
Lemma \ref{snegec} can be stated in different ways. For example, one of them reads as  that if
$$
P = \left[ \begin{matrix} 
P_2 & 0 & 0 & \ldots & 0 & 0\cr
0 & 0 & 0 & \ldots & 0 &0\cr
0 & 0 & P_2 & \ldots & 0 & 0\cr
\vdots & \vdots & \vdots & \ddots & \vdots & \vdots \cr  0 & 0 & 0 & \ldots & P_2 &0 \cr 0 & 0 & 0 & \ldots & 0&0 \cr\end{matrix} \right]\ \ \ {\rm and} \ \ \ 
Q = 
 \left[ \begin{matrix} 
Q_{q_1} & 0 & 0 & \ldots & 0 & 0\cr
0 & 0 & 0 & \ldots & 0 & 0\cr
0 & 0 & Q_{q_3} & \ldots & 0 & 0\cr
\vdots & \vdots & \vdots & \ddots & \vdots & \vdots \cr  0 & 0 & 0 & \ldots & Q_{q_k} & 0 \cr 0 & 0 & 0 & \ldots & 0 & 0 \cr\end{matrix} \right],
$$
where $Q_{q_i}$'s are as above, then every
matrix  $A \in M_n (\mathbb{C})$ that commutes with both $P$ and $Q$
is of the form
$$
A = \left[ \begin{matrix} 
\lambda_1 I & 0 & 0 & \ldots & 0 & 0\cr
0 & * & 0 & \ldots & 0 & *\cr
0 & 0 & \lambda_3 I & \ldots & 0 & 0\cr
\vdots & \vdots & \vdots & \ddots & \vdots & \vdots \cr  0 & 0 & 0 & \ldots & \lambda_k I & 0 \cr  0 & * & 0 & \ldots & 0 & * \cr \end{matrix} \right],
$$
where the $*$'s can be any matrices of  appropriate sizes. 
 Indeed, this follows  by  permuting the basis vectors in Lemma \ref{snegec}. 

We have pointed out just one concrete variation of Lemma \ref{snegec}. In the proof of \ref{jan2} we will also need some other variations which can be verified in the same straightforward manner.  
\end{remark}



\begin{lemma}\label{jan}
Let $n$  be an even positive integer,  and let
$$
R = \left[ \begin{matrix} I_{\frac{n}{2}} & 0 \cr 0 & 0 \cr \end{matrix} \right]  \in M_n (\mathbb{C}),
$$
where $I_{\frac{n}{2}}$ stands for the $\frac{n}{2}\times \frac{n}{2}$ identity matrix. Then there exists a unitary matrix $U \in M_n (\mathbb{C})$ such that every matrix
$A \in M_n (\mathbb{C})$ that commutes with both $R$ and $URU^\ast$ is a diagonal matrix.
\end{lemma}

\begin{proof}
Note that 
$$
R = S \,  \left[ \begin{matrix} 
P_2 & 0 & 0 & \ldots & 0 \cr
0 & P_2 & 0 & \ldots & 0 \cr
0 & 0 & P_2 & \ldots & 0 \cr
\vdots & \vdots & \vdots & \ddots & \vdots  \cr  0 & 0 & 0 & \ldots & P_2  \cr \end{matrix} \right]
\, S^\ast ,
$$
where $S$ is an appropriate permutation matrix. Note also that a matrix $A$ is diagonal if and only if the matrix $SAS^\ast$ is diagonal. 
Moreover, the matrices
$$
\left[ \begin{matrix} 
P_2 & 0 & 0 & \ldots & 0 \cr
0 & P_2 & 0 & \ldots & 0 \cr
0 & 0 & P_2 & \ldots & 0 \cr
\vdots & \vdots & \vdots & \ddots & \vdots  \cr  0 & 0 & 0 & \ldots & P_2  \cr \end{matrix} \right]
\ \ \ {\rm and} \ \ \
\left[ \begin{matrix} 
R_{q_1} & 0 & 0 & \ldots & 0 \cr
0 & R_{q_2} & 0 & \ldots & 0 \cr
0 & 0 & R_{q_3} & \ldots & 0 \cr
\vdots & \vdots & \vdots & \ddots & \vdots  \cr  0 & 0 & 0 & \ldots & R_{q_k}  \cr \end{matrix} \right]
$$
are unitarily similar.
Hence, the desired conclusion follows 
from Lemma \ref{snegec}.
\end{proof}

\begin{lemma}\label{jan2}
Let 
$$
R_1 = \left[ \begin{matrix} I_p & 0 & 0 \cr 0 & 0 & 0 \cr 0 & 0& 0 \cr\end{matrix} \right]  \ \ \ {\rm and} \ \ \  R_2 = \left[ \begin{matrix} 0 & 0 & 0 \cr 0 & I_q & 0 \cr 0 & 0& 0 \cr\end{matrix} \right]
$$
be $n\times n$ matrices,
where $I_p$ and $I_q$ stand for the $p\times p$ identity matrix and the $q\times q$ identity matrix, respectively. Assume that
\begin{equation}\label{prdprd}
p,q, n-p-q < {n \over 2}.
\end{equation}
  Then there exists a unitary matrix $U \in M_n (\mathbb{C})$ such that every
$A \in M_n (\mathbb{C})$ that commutes with each of the matrices $R_1, R_2, UR_1U^\ast, UR_2U^\ast$ is a diagonal matrix.
\end{lemma}

\begin{proof}
We will first show that, for any  permutation matrix
$S$, there is no loss of generality in replacing
  $R_1$ and $R_2$ by $SR_1 S^\ast$ and $SR_2 S^\ast$, respectively.  Note first that $S$ is 
 unitary and, for every matrix $A$, the matrix $SAS^\ast$ is diagonal if and only if $A$ is diagonal.
So, assume that there exists a unitary matrix $W$ and a permutation matrix $S$ such that each $B$ that commutes with each of the matrices 
\begin{equation}\label{mr1}S R_1 S^\ast ,S R_2 S^\ast , W(S R_1S^\ast) W^\ast, W(S R_2 S^\ast)W^\ast\end{equation} is a diagonal matrix.
We choose the unitary matrix $U= S^\ast W S$. Now, if $A$ commutes with each of the matrices $R_1, R_2, UR_1U^\ast, UR_2U^\ast$, then clearly $SAS^\ast$  commutes with each of the matrices 
listed in \eqref{mr1}.
Hence $SAS^\ast$ is diagonal, and consequently, $A$ is diagonal, as desired.

Let us start with the case where $n$ is even. After applying an appropriate permutation similarity and using (\ref{prdprd}) we may assume that
$$
R_1 = \left[ \begin{matrix} T_1 & 0 & 0 \cr 0 & T_2 & 0 \cr 0 & 0& 0 \cr\end{matrix} \right]  \ \ \ {\rm and} \ \ \  R_2 = \left[ \begin{matrix} 0 & 0 & 0 \cr 0 & S_1 & 0 \cr 0 & 0& S_2 \cr\end{matrix} \right],
$$
(the corresponding blocks in block matrix representations of $R_1$ and $R_2$ are of the same sizes)
where
$$
T_1 = \left[ \begin{matrix} 
P_2 & 0 & 0 & \ldots & 0 \cr
0 & P_2 & 0 & \ldots & 0 \cr
0 & 0 & P_2 & \ldots & 0 \cr
\vdots & \vdots & \vdots & \ddots & \vdots  \cr  0 & 0 & 0 & \ldots & P_2  \cr \end{matrix} \right] 
$$
is of  size  $2r_1 \times 2r_1$ where $r_1=\frac{n}{2}-q$,
$$T_2 = \left[ \begin{matrix} 
P_2 & 0 & 0 & \ldots & 0 \cr
0 & P_2 & 0 & \ldots & 0 \cr
0 & 0 & P_2 & \ldots & 0 \cr
\vdots & \vdots & \vdots & \ddots & \vdots  \cr  0 & 0 & 0 & \ldots & P_2  \cr \end{matrix} \right]
$$
and
$$
S_1 = \left[ \begin{matrix} 
I-P_2 & 0 & 0 & \ldots & 0 \cr
0 & I-P_2 & 0 & \ldots & 0 \cr
0 & 0 & I-P_2 & \ldots & 0 \cr
\vdots & \vdots & \vdots & \ddots & \vdots  \cr  0 & 0 & 0 & \ldots & I-P_2  \cr \end{matrix} \right]$$ 
are of  size $2r_2 \times 2r_2$ where $r_2 = p+q-\frac{n}{2}$,
and $$
S_2 = \left[ \begin{matrix} 
P_2 & 0 & 0 & \ldots & 0 \cr
0 & P_2 & 0 & \ldots & 0 \cr
0 & 0 & P_2 & \ldots & 0 \cr
\vdots & \vdots & \vdots & \ddots & \vdots  \cr  0 & 0 & 0 & \ldots & P_2  \cr \end{matrix} \right].
$$ is of  size $2r_3 \times 2r_3$ where $r_3=\frac{n}{2}-p$.


Choose  real numbers $q_1, \ldots , q_{r_1}, t_1, \ldots, t_{r_2}, s_1 , \ldots, s_{r_3}  \in (0,1)$  
such that  $q_1, \ldots , q_{r_1}, t_1, \ldots, t_{r_2}$ as well as
 $ 1-t_1, \ldots,1- t_{r_2}, s_1 , \ldots, s_{r_3}$ are pairwise different. Take $2\times 2$ unitary matrices
$U_1 , \ldots , U_{r_1} , V_1 ,\ldots, V_{r_2}, W_1 , \ldots W_{r_3}$  satisfying
$$
U_j P_2 U_{j}^\ast = R_{q_j}, \ \ \ j= 1, \ldots , r_1,
$$
$$
V_j P_2 V_{j}^\ast = R_{t_j}, \ \ \ j= 1, \ldots , r_2,
$$
$$
W_j P_2 W_{j}^\ast = R_{s_j}, \ \ \ j= 1, \ldots , r_3 .
$$
Let $U$ be the $n \times n$ unitary matrix that is block diagonal with $2\times 2$ unitary matrices $U_1 , \ldots , U_{r_1} , V_1 ,\ldots, V_{r_2}, W_1 , \ldots W_{r_3}$ as diagonal blocks. Then
$$
U R_1 U^\ast = \left[ \begin{matrix} K_1 & 0 & 0 \cr 0 & K_2 & 0 \cr 0 & 0& 0 \cr\end{matrix} \right]  \ \ \ {\rm and} \ \ \ U R_2 U^\ast  = \left[ \begin{matrix} 0 & 0 & 0 \cr 0 & L_1 & 0 \cr 0 & 0& L_2 \cr\end{matrix} \right],
$$
where
$$
K_1 = \left[ \begin{matrix} 
R_{q_1} & 0 & 0 & \ldots & 0 \cr
0 & R_{q_2} & 0 & \ldots & 0 \cr
0 & 0 & R_{q_3} & \ldots & 0 \cr
\vdots & \vdots & \vdots & \ddots & \vdots  \cr  0 & 0 & 0 & \ldots & R_{q_{r_1}}  \cr \end{matrix} \right], \ \ \
K_2 = \left[ \begin{matrix} 
R_{t_1} & 0 & 0 & \ldots & 0 \cr
0 & R_{t_2} & 0 & \ldots & 0 \cr
0 & 0 & R_{t_3} & \ldots & 0 \cr
\vdots & \vdots & \vdots & \ddots & \vdots  \cr  0 & 0 & 0 & \ldots & R_{t_{r_2}}  \cr \end{matrix} \right],
$$
and
$$
L_1 = \left[ \begin{matrix} 
R_{1-t_1}^- & 0 & 0 & \ldots & 0 \cr
0 & R_{1-t_2}^- & 0 & \ldots & 0 \cr
0 & 0 & R_{1-t_3}^- & \ldots & 0 \cr
\vdots & \vdots & \vdots & \ddots & \vdots  \cr  0 & 0 & 0 & \ldots & R_{1-t_{r_2}}^-  \cr \end{matrix} \right],$$ 
 $$
L_2 = \left[ \begin{matrix} R_{s_1} & 0 & 0 & \ldots & 0 \cr
0 & R_{s_2} & 0 & \ldots & 0 \cr
0 & 0 & R_{s_3} & \ldots & 0 \cr
\vdots & \vdots & \vdots & \ddots & \vdots  \cr  0 & 0 & 0 & \ldots & R_{s_{r_3}}  \cr \end{matrix} \right].
$$

Assume now that $A \in M_n (\mathbb{C})$ commutes with  the matrices $R_1, R_2, UR_1U^\ast$, and $UR_2U^\ast$. Using Lemma \ref{snegec}, we see that
 $A R_1 = R_1 A$ and $A (UR_1 U^\ast) = (UR_1 U^\ast) A$ imply that
$$
A =  \left[ \begin{matrix}  D_1 & 0 & 0 \cr 0 & D_2 & 0 \cr 0 & 0 & *  \cr \end{matrix} \right]
$$
for some diagonal matrices $D_1$ and $D_2$ and some matrix $*$. Similarly, the other two commutativity conditions yield
that
$$
A =  \left[ \begin{matrix}  * & 0 & 0 \cr 0 & D_3 & 0 \cr 0 & 0 & D_4  \cr \end{matrix} \right]
$$
for some diagonal matrices $D_3$ and $D_4$ and some matrix $*$. Hence, $A$ is diagonal, as desired.

It remains to consider the case where $n$ is odd. 
 After applying an appropriate permutation similarity and using (\ref{prdprd}) we can assume that
$$
R_1 = \left[ \begin{matrix} T_0 & 0 & 0 &0 \cr 0& T_1 & 0 & 0 \cr 0& 0 & T_2 & 0 \cr0&  0 & 0& 0 \cr\end{matrix} \right]  \ \ \ {\rm and} \ \ \  R_2 = \left[ \begin{matrix}  S_0 & 0 & 0 &0 \cr 0& 0 & 0 & 0 \cr 0& 0 & S_1 & 0 \cr 0& 0 & 0& S_2 \cr\end{matrix} \right],
$$
(the corresponding blocks in block matrix representations of $R_1$ and $R_2$ are of the same sizes; some rows and the corresponding columns may be  absent)
where $T_1$, $T_2$, $S_1$, and $S_2$ are as above, and
$$
T_0 =  \left[ \begin{matrix}  1 & 0 & 0 \cr 0 & 0 & 0 \cr 0 & 0 & 0 \cr \end{matrix} \right] \ \ \ {\rm and} \ \ \ 
S_0 =  \left[ \begin{matrix}  0 & 0 & 0 \cr 0 & 1 & 0 \cr 0 & 0 & 0 \cr \end{matrix} \right] .
$$

Set
$$
K_0 =  \left[ \begin{matrix}  {1 \over 2} & {1 \over 2} & 0 \cr {1 \over 2} & {1 \over 2} & 0 \cr 0 & 0 & 0 \cr \end{matrix} \right] \ \ \ {\rm and} \ \ \ 
L_0 =  \left[ \begin{matrix}  {1 \over 3} &  -{1 \over 3} &  {1 \over 3} \cr - {1 \over 3} &  {1 \over 3} & - {1 \over 3} \cr  {1 \over 3} & -  {1 \over 3} &  {1 \over 3} \cr \end{matrix} \right] .
$$
Observe that $K_0$ and $L_0$ are, just as $T_0$ and $S_0$, a pair of orthogonal projections of rank one. Therefore, there exists  a unitary $3\times 3$ matrix $U_0$ such that
$$
U_0 T_0 U_{0}^\ast = K_0 \ \ \mbox{and} \ \ U_0 S_0 U_{0}^\ast = L_0 .
$$
Consequently, there exists
a unitary  $n\times n$ matrix $U$ such that
$$
U R_1 U^\ast = \left[ \begin{matrix}K_0 & 0 & 0 & 0 \cr 0&  K_1 & 0 & 0 \cr 0& 0 & K_2 & 0 \cr 0& 0 & 0& 0 \cr\end{matrix} \right]  \ \ \ {\rm and} \ \ \ U R_2 U^\ast  = \left[ \begin{matrix} L_0 & 0 & 0 & 0 \cr 0 & 0 & 0 & 0 \cr 0&  0 & L_1 & 0 \cr 0 & 0 & 0& L_2 \cr\end{matrix} \right],
$$
where 
 $K_1$, $K_2$, $L_1$, and $L_2$ are as above with the additional requirement that all the real numbers  $q_1, \ldots , q_{r_1}, t_1, \ldots, t_{r_2}, s_1 , \ldots, s_{r_3}$, and $ 1-t_1, \ldots,1- t_{r_2}$ are different from $\frac{1}{2}$ and $\frac{1}{3}$. 

Now take an $n \times n$ matrix $A$ that commutes with  each of the matrices $R_1,R_2,UR_1U^\ast$, and $UR_2U^\ast$. Our goal is to show that $A$ is a diagonal.


As  $A$ commutes with $R_1$ and $UR_1 U^\ast$,
it follows from Lemma \ref{snegec}, along with Remark \ref{zgoljdodatnarazlags}, that, in particular, $A$ is of the form
$$
A=  \left[ \begin{matrix} * & 0 & 0 &* \cr 0& D' & 0 & 0 \cr 0& 0 & D'' & 0 \cr *&  0 & 0& * \cr\end{matrix} \right],
$$
where the matrices $D'$ and $D''$ are diagonal.

Similarly, we claim that the condition that 
 $A$ commutes with $R_2$ and $UR_2 U^\ast$ implies that $A$ is of the form
$$
A=  \left[ \begin{matrix} * & * & 0 &0 \cr * & * & 0 & 0 \cr 0& 0 & D''' & 0 \cr 0&  0 & 0& D'''' \cr\end{matrix} \right]
$$
for some diagonal matrices $D'''$ and $D''''$. Indeed, although our situation is, due to the upper $3\times 3$ corner, not exactly the same as the one considered  in Remark \ref{zgoljdodatnarazlags},
it is the same up to unitary similarity. As our conclusion concerns only the form of the lower corners, this is enough. 

The last two paragraphs together show that 
$$
A=  \left[ \begin{matrix} * & 0 & 0 &0 \cr 0 & D' & 0 & 0 \cr 0& 0 & D''' & 0 \cr 0&  0 & 0& D'''' \cr\end{matrix} \right].
$$
Since only diagonal $3\times 3$ matrices commute with both  $S_0$ and $T_0$, and since 
 $A$ commutes with  $R_1$ and $R_2$, it follows that the upper $3\times 3$ corner of $A$ is also a diagonal matrix.
\end{proof} 

We continue with two similar lemmas. The first one will be derived from  Lemma \ref{jan}  and the second one from Lemma \ref{jan2}. First we need some notation.

By  $M_{n}^0 (\mathbb{C})$ we denote the set of all $n\times n$ complex matrices  all of whose diagonal entries are equal to zero.

Assume now that $n$ is  even. Denote by $\V_{\frac{n}{2}}$ the linear space  of all $n\times n$ complex matrices of the form
$$
\left[ \begin{matrix}  0 & \ast \cr \ast & 0 \cr \end{matrix} \right] ,
$$
where all block matrices are of size $\frac{n}{2}\times \frac{n}{2}$; of course, $0$ stands for the zero matrix and $*$ for any matrix.

\begin{lemma} \label{nova1}Let $n$ be an even integer. There exists a unitary  matrix $U\in M_n(\C)$ such that
$$M_{n}^0 (\mathbb{C}) \subseteq \V_{\frac{n}{2}} + U \V_{\frac{n}{2}} U^*.$$
\end{lemma}

\begin{proof}
Recall that 
\begin{equation} \label{inn}\langle X, Y \rangle = {\rm tr}\, (XY^\ast)\end{equation}
defines an inner product on 
$M_{n} (\mathbb{C})$. 
Note that $\langle X,Y \rangle = \sum_{i,j=1}^n  x_{ij}\overline{y_{ij}}$, where $X=(x_{ij})$ and $Y=(y_{ij})$, so this is actually the standard inner product on $M_n (\mathbb{C})$ (when identified with  $\mathbb{C}^{n^2}$).  Writing $\V$ for $\V_{\frac{n}{2}}$ for brevity, we thus have that
 with respect to this inner product, 
$
{\V}^\perp$ is the linear space of all matrices of the form
$$
\left[ \begin{matrix}   \ast & 0 \cr 0 & \ast  \cr \end{matrix} \right] .
$$

Let
$$
R = \left[ \begin{matrix} I_{\frac{n}{2}} & 0 \cr 0 & 0 \cr \end{matrix} \right].
$$
Observe that for every matrix $M \in M_{n} (\mathbb{C})$, 
$$ M\in {\V}^\perp \iff
RM = MR.
$$
Hence,  for any
unitary matrix $U$, 
$$
M \in U {\V}^\perp U^\ast \iff M (URU^\ast) =   (URU^\ast) M.
$$
The map $X \mapsto UXU^\ast$ is a unitary operator on $M_{n} (\mathbb{C})$. Indeed, for every pair $X,Y \in M_{n} (\mathbb{C})$ we have
$$
\langle  UXU^\ast,  UYU^\ast \rangle = {\rm tr}\, (  UXU^\ast ( UYU^\ast)^\ast) = {\rm tr}\, (XY^\ast) = \langle X,Y \rangle .
$$
It follows that
$$
 U {\V}^\perp U^\ast =  (U {\V} U^\ast)^\perp .
$$

Assume now that $U$ is the matrix from Lemma \ref{jan}. Then every matrix $M \in {\V}^\perp \cap  (U {\V} U^\ast)^\perp$ is diagonal.
Hence, $$({\V} +  U {\V} U^\ast)^\perp={\V}^\perp \cap  (U {\V} U^\ast)^\perp   \subseteq D_n (\mathbb{C}),$$ where $ D_n (\mathbb{C})$ denotes the set of all diagonal matrices.
It follows that
$$
M_{n}^0 (\mathbb{C}) = D_n (\mathbb{C})^\perp \subseteq{\V} +  U {\V} U^\ast,
$$
which completes the proof.
\end{proof}

Let $p,q,r$ be positive integers and let $n=p+q+r$ ($n$ can now be either even or odd). 
Denote by
 $\V_{p,q,r}$  the set of all $n\times n$ matrices of the form
$$
\left[ \begin{matrix}  0 & \ast & \ast \cr  \ast & 0 & \ast \cr \ast & \ast & 0 \cr \end{matrix} \right],
$$
where the diagonal blocks are zero matrices of sizes 
$p\times p$, $q\times q$, and $r\times r$, respectively, and  the $*$'s stand for any matrices of  appropriate sizes.

\begin{lemma} \label{nova2}Let $n$ be a positive integer and let $p,q,r$ be positive integers such that $n=p+q+r$ and $p,q,r < \frac{n}{2}$.
 There exists a unitary matrix $U\in M_n(\C)$ such that
$$M_{n}^0 (\mathbb{C}) \subseteq \V_{p,q,r} + U\V_{p,q,r} U^*.$$
\end{lemma}

\begin{proof}
As in the preceding proof, consider $M_n(\C)$ as the inner product space with respect to the inner product \eqref{inn}.
Observe that
 ${\mathcal V}^\perp$ 
is the linear space of all matrices of the form
$$
\left[ \begin{matrix}  \ast & 0 & 0 \cr  0 & \ast & 0  \cr 0 & 0 & \ast  \cr \end{matrix} \right] .
$$
Similarly as in the preceding proof we observe that a matrix $M \in M_{n} (\mathbb{C})$ belongs to ${\mathcal V}^\perp$ if and only if 
$$
R_1 M = MR_1 \ \ \ {\rm and} \ \ \ R_2 M = M R_2,
$$
where
$$
R_1 = \left[ \begin{matrix} I_p & 0 & 0 \cr 0 & 0 & 0 \cr 0 & 0 & 0 \cr\end{matrix} \right] \ \ \ {\rm and} \ \ \ 
R_2 = \left[ \begin{matrix} 0 & 0 & 0 \cr 0 & I_q & 0 \cr 0 & 0 & 0 \cr\end{matrix} \right] .
$$
If we take the unitary matrix $U$ from
Lemma \ref{jan2}, we see that every matrix $M \in {\mathcal V}^\perp \cap  (U {\mathcal V} U^\ast)^\perp$ is diagonal. We  then  complete the proof in exactly the same way as in the proof of Lemma
\ref{nova1}.
\end{proof}

The last lemma we need is an easy consequence of
the Sylvester-Rosenblum theorem---see the discussion in
\cite[Section 3]{BR}.

\begin{lemma}\label{vohka}
Let $n_1,\dots,n_k$ be positive integers, let $n=n_1+\dots + n_k$, and let $A$ an $n\times n$ block diagonal matrix,
$$
A = \left[ \begin{matrix} 
A_1 & 0 & 0 & \ldots & 0 \cr
0 & A_2 & 0 & \ldots & 0 \cr
0 & 0 & A_3 & \ldots & 0 \cr
\vdots & \vdots & \vdots & \ddots & \vdots  \cr  0 & 0 & 0 & \ldots & A_k  \cr \end{matrix} \right],
$$
where $A_i$ is an $n_i\times n_i$ matrix and the zeros are matrices of appropriate sizes.
If the spectra of $A_1,A_2,\dots,A_k$  are pairwise disjoint, then
 $A$ is similar to
any matrix of the form
$$
\left[ \begin{matrix} 
A_1 & * & * & \ldots & * \cr
0 & A_2 & * & \ldots & * \cr
0 & 0 & A_3 & \ldots & * \cr
\vdots & \vdots & \vdots & \ddots & \vdots  \cr  0 & 0 & 0 & \ldots & A_k  \cr \end{matrix} \right],
$$
as well as to
any matrix of the form
$$
\left[ \begin{matrix} 
A_1 & 0 & 0 & \ldots & 0 \cr
* & A_2 & 0 & \ldots & 0 \cr
* & * & A_3 & \ldots & 0 \cr
\vdots & \vdots & \vdots & \ddots & \vdots  \cr  * & * & * & \ldots & A_k  \cr \end{matrix} \right],
$$
where the $*$'s stand for any matrices of  appropriate sizes.
\end{lemma}

We are now in a position to prove the theorem to which we have been heading.

\begin{theorem}\label{taglavn}
Assume that each eigenvalue of the matrix $B \in M_n (\mathbb{C})$ has algebraic multiplicity $\le \frac{n}{2}$. Then every trace zero matrix $A\in M_n(\C)$ can be written as 
$$
A = B'-B''+B'''-B'''',
$$
where each of $B',B'',B''',B''''$ is similar to $B$.
\end{theorem}

\begin{proof}
Since every trace zero matrix is  similar
to a matrix in $M_{n}^0 (\mathbb{C})$
\cite[Problem 3 on p.\,77]{HJ}, there is no loss of generality in assuming that $A\in M_{n}^0 (\mathbb{C})$.

The matrix $B$ satisfies one of the conditions (a) or (b) of Lemma \ref{brsvt}. 

Suppose first that $B$ satisfies (b).
Then $B$ is similar to a block diagonal matrix
\begin{equation}\label{kaeng}
B = \left[ \begin{matrix} B_1 & 0 &0 \cr 0 & B_2 & 0\cr 0 & 0 & B_3 \end{matrix} \right],
\end{equation}
where the diagonal blocks are of sizes $p \times p$, $q\times q$, and $r\times r$ with
$
p,q,r < {n \over 2}$,
and the spectra of $B_1$ and $B_2$ and $B_3$ are pairwise disjoint. Take any $C \in \V_{p,q,r}$. Then
$$
C = \left[ \begin{matrix}  0 & C_1 & C_2 \cr  C_3 & 0 & C_4 \cr C_5 & C_6 & 0 \cr \end{matrix} \right] 
$$
for some matrices $C_i$ of appropriate sizes. By Lemma \ref{vohka}, both matrices 
$$
B'= \left[ \begin{matrix}  B_1 & C_1 & C_2 \cr  0 & B_2 & C_4 \cr 0 & 0 & B_3 \cr \end{matrix} \right] \ \ \ {\rm and} \ \ \  B''= \left[ \begin{matrix}  B_1 & 0 & 0 \cr  -C_3 & B_2 & 0 \cr -C_5 & -C_6 & B_3 \cr \end{matrix} \right]
\ \ 
$$
are similar to $B$, and clearly,
$$
C = B'-B''.
$$
Thus, every matrix in $\V_{p,q,r}$ is a difference of two matrices that are similar to $B$.
The desired conclusion therefore follows from Lemma \ref{nova2}.

The case where $B$ satisfies (a) can be handled similarly. One just has to apply Lemma \ref{nova1} instead of Lemma \ref{nova2}, and use Lemma \ref{vohka} for $k=2$ rather than for 
$k=3$.
\end{proof}

\begin{remark} Some restriction on the  multiplicity of eigenvalues of the matrix $B$ is definitely needed in order to arrive at the conclusion of the theorem. For example,
if $B$ has an eigenvalue $\lambda$ whose geometric multiplicity is  $> \frac{3n}{4}$, then for any invertible matrices $T_1,T_2,T_3,T_4$,
\begin{align*}
&T_1BT_1^{-1} - T_2BT_2^{-1} + T_3BT_3^{-1} - T_4BT_4^{-1} \\
 = & T_1(B-\lambda I)T_1^{-1} - T_2(B-\lambda I)T_2^{-1} + T_3(B-\lambda I)T_3^{-1} - T_4(B-\lambda I)T_4^{-1}  \end{align*}
is never an invertible matrix.
\end{remark}

\begin{remark}\label{conj}
Let $\A$ be any algebra and let $f=f(X_1,\dots,X_m)$ be any polynomial. If $t$ is an invertible element in $\A$, then
$$tf(a_1,\dots,a_m)t^{-1} = f(ta_1t^{-1},\dots,ta_mt^{-1}).$$
This shows that $f(\A)$ is closed under conjugation by invertible elements.
\end{remark}

Proving  our first main theorem is now straightforward.

\bigskip

\noindent
{\em Proof of Theorem \ref{main}}.
 By \cite[Theorem 3.8]{B}, $f(\A)$ contains a matrix  all of whose  eigenvalues have
algebraic multiplicity  $\le \frac{n}{2}$. Remark \ref{conj} tells us that any matrix that is similar to $B$
also lies in $f(\A)$.
 The desired conclusion therefore follows from Theorem \ref{taglavn}.
\hfill\(\Box\)

\section{Proof of Theorem \ref{main2}} \label{lg}

Theorem \ref{main2} will be derived from the next lemma which concerns central simple algebras. By such an algebra we mean, as usual, a simple 
algebra whose center consists of scalar multiples of unity, and is also finite-dimensional---however, see Remark \ref{rfd}.

 We write 
$C_\A(a)$ 
for the centralizer of the element $a$ in the algebra 
$\A$, i.e., $C_\A(a)=\{x\in \A\,|\, xa=ax\}$. Note that if $\A=M_2(F)$, $C_\A(a) = {\rm span}\,\{1,a\}$ for every nonscalar matrix $a$. In particular, $C_\A(a)$ is a commutative subalgebra.

\begin{lemma}\label{l}
Let $F$ be any field, let $\A$ be a central simple  $F$-algebra, and let $\U =\{u\in \A\,|\,u^2 \in F1\}$.
Suppose there exists a $d\in \A$ such that
\begin{enumerate}
\item[{\rm (a)}] $d$ is a sum of two elements from $\U$, and
\item[{\rm (b)}] $C_\A(d^2)$ is a commutative subalgebra.
\end{enumerate}
Then $\dim_F \A = 1$ or $\dim_F \A = 4$.
\end{lemma}

\begin{proof}
Let $u,v\in \U$   be such that $d= u + v$. Squaring both sides we obtain $d^2 - (uv + vu)\in F1$. Hence, $C_\A(uv + vu)=C_\A(d^2)$ is a commutative subalgebra.
Note that $u,v\in\U$ implies  that $u,v \in C_\A(uv + vu)$, and so $uv=vu$. Therefore, $w=2uv =uv+vu\in \U$. 

Let $\bar{F}$ be the algebraic closure of $F$ and let $\bar{\A} = \bar{F}\otimes_F \A$ be the scalar extension of $\A$ to $\bar{F}$. By a standard application of Wedderburn's theorem, we may assume that 
$\bar{\A}= M_n(\bar{F})$ where $n = \sqrt{\dim_F \A}$. Let
$\bar{w}=1\otimes w$. Take $c\in C_{\bar{\A}}(\bar{w})$. Writing
$c=\sum_i \lambda_i \otimes x_i$ with the $\lambda_i$'s  linearly independent over $F$, we have $\sum_i \lambda_i \otimes [x_i,w]=0$ and hence $[x_i,w]=0$ for every $i$. This shows that
$C_{\bar{\A}}(\bar{w}) = \bar{F} \otimes_F C_\A(w) $. Since, by the first paragraph, the algebra $C_\A(w)$ is commutative, so is $C_{\bar{\A}}(\bar{w})$.

Since $w\in\U$,  the square of $\bar{w}$  is a scalar matrix. Hence, either $\bar{w}^2 =0$ or $\bar{w}$ is a diagonalizable matrix having at most two  distinct eigenvalues. In each of the cases, we see  from the Jordan normal of $\bar{w}$ that $C_{\bar{\A}}(\bar{w})$ contains a noncommuting pair of matrices   if $n\ge 3$. Therefore, $n\le 2$.
\end{proof}

\begin{remark}\label{rfd} The assumption that $\A$ is finite-dimensional is actually redundant. 
Indeed, the first paragraph of the proof shows that  for any algebra $\A$, the assumptions of the lemma imply that $\A$ contains an element $w$ such that 
$w^2\in F1$ and $C_\A(w)$ is commutative. Since $wx+xw\in C_\A(w)$ for every $x\in \A$, it follows that
$[wx+xw, yw + wy]=0$ for all $x,y\in\A$, that is, \begin{equation}\label{eqa}
(wx+xw)yw + (wxw+xw^2)y= y(w^2x+wxw) + wy(wx+xw).\end{equation}
 Assume now that $\A$ is simple and that its center equals $F 1$. Our goal is to show that $\A$ is then finite-dimensional. We may assume that $w\notin F1$. Using \cite[Theorem 7.43]{INCA}, we see that \eqref{eqa} implies that $wx+xw\in {\rm span}\, \{1,w\}$ 
for every $x\in\A$. Take $t\in\A$ such that $[t,w]\ne 0$. From
$$x[t,w] = \big(w(xt) + (xt)w\big) - (wx + xw)t,$$
along with $w(xt) + (xt)w, wx+xw\in {\rm span}\, \{1,w\}$, we see that $x[t,w]$ lies in $\mathcal T={\rm span}\, \{1,w,t,wt\}$ for every $x\in\A$. By the simplicity of $\A$, there exist
$a_i,b_i\in \A$ such that $\sum_i a_i [t,w]b_i =1$. Now, $xa_i[t,w]b_i\in \mathcal T b_i$ and so  $\A$ is equal to $\sum_i \mathcal T b_i$, which is a finite-dimensional space.
\end{remark}

 We actually need Lemma \ref{l} only in the case where $\A=M_n(F)$. However, because of the method of proof it was more natural to state it for general central simple algebras. Besides, 
 together with Remark \ref{rfd}, this lemma
 may be of independent interest. 

\bigskip

\noindent
{\em Proof of Theorem \ref{main2}}. 
 Since $n > 2$ and $F$ has characteristic $0$, we can choose distinct $\lambda_1,\dots,\lambda_n\in F$ such that $\lambda_1+\dots+\lambda_n =0$ and $\lambda_i\ne -\lambda_j$ whenever $i\ne j$.
 Let $D$ be the diagonal matrix with diagonal entries $\lambda_1,\dots,\lambda_n$. Note that $D$ has trace zero
and that $D^2$ is a diagonal matrix with diagonal entries $\lambda_1^2,\dots,\lambda_n^2$ which are all distinct. Therefore, only diagonal matrices commute with $D^2$. In particular,
the centralizer $C_\A(D^2)$ is a commutative algebra.

If the theorem was false, there would exist $A,B\in f(\A)$ such that $D=\alpha A + \beta B$ for some $\alpha,\beta \in F$. Since $f$ is $2$-central, the squares of  $\alpha A$ and $\beta B$ are scalar matrices.  Hence $n\le 2$ by Lemma \ref{l}, a contradiction.
\hfill\(\Box\)

\begin{remark} If $f$ is $2$-central for $M_n(F)$, then $n$ is even and from \cite[Corollary 3.3]{S} it follows, 
under the additional assumption that $F$ is algebraically closed, that 
every matrix in the image of the polynomial $f^3$ is a scalar multiple of  a matrix
similar to
$$
 \left[ \begin{matrix} I_{\frac{n}{2}} & 0 \cr 0 & -I_{\frac{n}{2}} \cr \end{matrix} \right] ,
$$
where $I_{\frac{n}{2}}$ is the $\frac{n}{2}\times \frac{n}{2}$ identity matrix (and so, in particular, has zero trace). This sheds some light on the context of the preceding section.
\end{remark}

\section{Related results and open problems}\label{sr}

In this section, we record some results related to Theorems \ref{main} and \ref{main2}  and state some open problems.

\subsection{The cases where $f$ is multilinear or $n$ is prime}
Our first result presents two conditions under which $f(\A)-f(\A)$
already contains all trace zero matrices. The proof uses some standard notions and results from the theory of polynomial identities.

\begin{theorem}\label{mp}
 Let $F$ be an algebraically closed field of characteristic $0$, let
$\A=M_n(F)$,  and let 
 $f\in  F\langle \mathcal X\rangle$ be a polynomial which  is neither an identity nor a central polynomial of $\A$. Suppose  either
\begin{enumerate}
\item[{\rm (a)}] $n$ is prime, or
\item[{\rm (b)}] $f$ is multilinear.
\end{enumerate}
 Then $f(\A)-f(\A)$ contains all  trace zero matrices in $\A$.
\end{theorem}

\begin{proof}  Assume that $n$ is prime. Denote by
 ${\rm GM}_n(F)$ the algebra of $n\times n$ generic matrices, and  by ${\rm UD}_n(F)$ the universal division algebra of degree $n$ (i.e., the algebra of central quotients of ${\rm GM}_n(F)$).
Further, let $\overline{f}$ denote the image of the natural epimorphism from $ F\langle \mathcal X\rangle$ onto  ${\rm GM}_n(F)$. Our assumptions on $f$ imply that  $\overline{f}$ does not belong to the
 center  of 
${\rm GM}_n(F)$, and hence neither to the center
$Z$ of ${\rm UD}_n(F)$. Thus, $Z\subsetneq Z(\overline{f})\subsetneq {\rm UD}_n(F)$.
Since $n$ is prime and $[{\rm UD}_n(F):Z]=n^2$, it follows that $[Z(\overline{f}):Z]=n$. This implies that the characteristic polynomial $h$ of $\overline{f}$ is the minimal polynomial, and is hence irreducible over $Z$. As we work in characteristic $0$, $h$ has $n$ distinct roots in the
algebraic closure of $Z$, so its discriminant Disc$(h)$ is a nonzero element of $Z$. Therefore,   Disc$(h)(A_1,\dots,A_m)\ne 0$ for some matrices $A_i\in\A$.
Evaluating the coefficients of $h$ at this $m$-tuple $(A_1,\dots, A_m)$, we thus obtain a polynomial  with $n$ distinct roots  which is the characteristic polynomial of the matrix 
$f(A_1,\dots,A_m)$. Thus, $f(A_1,\dots,A_m)$ has $n$ distinct eigenvalues.

If $f$ is multilinear, the same conclusion that $f(\A)$ contains a matrix  with $n$ distinct eigenvalues follows from \cite[Theorem 1]{BMR}.

As $f(\A)$ is closed under conjugation by invertible matrices (Remark \ref{conj}), it follows that, in each of the two cases (a) and (b), $f(\A)$ contains a diagonal matrix
$D$ with distinct eigenvalues $\lambda_1,\dots,\lambda_n$ on the diagonal. Every upper triangular matrix with $\lambda_1,\dots,\lambda_n$ on the diagonal is similar to $D$ and is 
 therefore also contained in $f(\A)$, and the same is true for every lower triangular matrix with $\lambda_1,\dots,\lambda_n$ on the diagonal. Consequently, $f(\A)-f(\A)$ contains 
$M_{n}^0 (F)$,
the set of all  matrices in $\A$ all of whose diagonal entries are  zero. The desired conclusion therefore follows from the fact  (already used at the 
beginning of the proof of Theorem \ref{taglavn}) that every trace zero matrix is similar to a matrix in $M_{n}^0 (F)$.
\end{proof}

The argument given in the first paragraph of the proof was shown to us by Jurij Vol\v ci\v c, and the argument in the last paragraph was observed by Daniel Vitas.
We thank both for   allowing us to use their findings in this paper.

The L'vov-Kaplansky conjecture states that the image of a multilinear polynomial $f$ in $\A=M_n(F)$, where $F$ is an infinite field, is either $\{0\}$, the set of scalar matrices, the set
of trace zero matrices, or the whole set $\A$. Thus, if $f$ is neither an identity nor a central polynomial, then $f(\A)$ itself, not only $f(\A)-f(\A)$, supposedly contains all trace zero matrices. 
However, although Theorem \ref{mp} is far from the L'vov-Kaplansky conjecture (and, besides, the proof was an easy application of  \cite[Theorem 1]{BMR}), it may still be of interest since this conjecture is largely open. See the survey paper \cite{BMRs} by the main contributors in this area of research.

By \cite[Theorem 3.2 and Proposition 4.1]{BGKP}, the assertion that $f(\A)-f(\A)$ contains all  trace zero matrices in $\A =M_n(F)$ also holds for Lie polynomials $f$ satisfying certain conditions.
We are thankful to Urban Jezernik for drawing our attention to this result.

It thus may be said that for most polynomials $f$, the set
$f(\A)-f(\A)$
already includes all trace zero matrices. The situation occurring in Theorem \ref{main2} is quite exceptional.

\subsection{The Waring problem for $\mathcal B(H)$} The Waring problem, discussed above for the algebra $M_n(F)$, also makes sense for some infinite-dimensional algebras. In fact, as explained in \cite{B},
the motivation for its consideration  partially arose from the results on  presenting operators in $\mathcal B(H)$, the algebra of all bounded linear operator on an  infinite-dimensional separable (complex) Hilbert space $H$, as sums
of some  special-type operators. 

It was shown in \cite[Corollary 3.17]{B} that for any  nonconstant polynomial  $f\in \mathbb C\langle \mathcal X\rangle$, every operator in
 $\A=\mathcal B(H)$ is 
a sum of $3916$ operators from $f(\A)-f(\A)$. We will now prove that the number of summands can be significantly lowered by using the result by Davidson and Marcoux \cite[Theorem 2.15]{DM} (of which the first author was unaware at the time of publication of \cite{B}). This result states that if $B\in \mathcal B(H)$ is not a sum of a scalar operator and a compact operator, then every operator in $\mathcal B(H)$ can be written as  as a sum of at most eight operators similar to $B$.

\begin{theorem}\label{bh} Let $H$ be an infinite-dimensional separable Hilbert space and
let  $f\in \mathbb C\langle \mathcal X\rangle$ be a  nonconstant polynomial. Then every operator in
 $\A=\mathcal B(H)$ is 
a sum of eight or fewer operators from $f(\A)$. 
\end{theorem}

\begin{proof}
Take
 $n\ge 2$ such that $f$  is neither an identity nor a central polynomial of $M_n(\C)$. We may identify  $\A$ with $M_n(\A)$ and accordingly  consider a subalgebra $\mathcal M_n$ of $\A$ isomorphic 
to $M_n(\C)$ (with matrix units being idempotent and square-zero operators  having infinite dimensional kernel and range).  By \cite[Theorem 3.8]{B}, $f(\mathcal M_n)$ contains a matrix $B$ such that 
all of its eigenvalues
have  algebraic multiplicity  $\le \frac{n}{2}$. Note that $B$, viewed as an operator in $\mathcal B(H)$, is not a sum of a scalar operator and a compact operator. By Remark \ref{conj},
$f(\A)$ is closed under conjugation by invertible operators, and so the desired result follows from \cite[Theorem 2.15]{DM}.
\end{proof}

Most probably, this theorem is not  definitive.

\begin{question} 
What is the smallest number  that can replace ``eight'' in Theorem \ref{bh}?
\end{question} 

We remark that by considering linear combinations instead of sums, we can use \cite[Theorem 2.13]{DM} to replace ``eight'' by ``six''. However, it is hard to believe that this is the smallest number.
What is clear is that ``one'' is not enough since the identity operator $I$ is not a commutator (= the element in the image of the polynomial $[X_1,X_2]$ in $\mathcal B(H)$). On the other hand, the classical Halmos' theorem \cite{H} states  that every $T\in \mathcal B(H)$ is a sum of two commutators.  The most optimistic conjecture would thus be that ``eight'' can be replaced by ``two''.

\subsection{Problems on linear combinations} So far, we have considered  sums and differences of elements in $f(\A)$. This may be more useful for possible applications than  linear combinations of these elements (see \cite[Corollary 3.23]{B} as a sample application). Nevertheless, it is natural to ask whether the above results can be improved if we omit the condition that the 
coefficients of linear combinations may only  be $1$ and $-1$.

Theorems \ref{main} and \ref{main2} imply that if $f$
 is neither an identity nor a central polynomial of $\A=M_n(\C)$, then every trace zero matrix  is a linear combination of four, but not necessarily two, elements of $f(\A)$.  
The following question obviously occurs. 

\begin{question} If $f$
 is neither an identity nor a central polynomial of $\A=M_n(\C)$, is then every trace zero matrix in $\A$ a linear combination of three elements from $f(\A)$?
\end{question}

Polynomials $f,g\in F\langle \mathcal X\rangle$ are said to be {\em cyclically equivalent} if $f-g$ is a sum of commutators in 
$F\langle \mathcal X\rangle$. From \cite[Theorem 4.5]{BK} it follows that if a polynomial $f$ is not cyclically equivalent to an identity of $\A=M_n(F)$ (where $F$ is any field of characteristic $0$) and is not central 
for $\A$,
then the linear span of $f(\A)$ is equal to $\A$. 

Theorem \ref{main} yields the following corollary, which slightly improves \cite[Corollary 3.20]{B}.

\begin{corollary}\label{cs2}If a polynomial $f\in F\langle \mathcal X\rangle$ is  
 not cyclically equivalent to an identity of $\A=M_n(\C)$
and is
not  central for $\A$, then every matrix in $\A$ is a linear combination of five matrices from $f(\A)$.
\end{corollary}

\begin{proof} By the result mentioned before the statement of the corollary, $f(\A)$ contains a matrix $A$ with nonzero trace. Writing $T\in \A$ as
$$T= \frac{{\rm tr}(T)}{{\rm tr}(A)}A + \Bigl(T - \frac{{\rm tr}T)}{{\rm tr}(A)}A\Bigr)$$
and applying  Theorem \ref{main} to the trace zero matrix $T - \frac{{\rm tr}(T)}{{\rm tr}(A)}A$ we obtain the desired conclusion.
\end{proof}

 Theorem \ref{main2} tells us that ``five'' cannot be replaced by ``two''. However, the following question remains open.

\begin{question} Can ``five'' in Corollary \ref{cs2} be replaced by ``four'' or even ``three''?
\end{question}

\subsection{Problem on fields and other simple algebras} 
In \cite{C}, Chuang proved that if  
 $F$ is a finite field, then a subset $\mathcal S$ of $M_n(F)$  is the image of a polynomial with zero constant term  if and only if $\mathcal S$ contains $0$ and $T\mathcal S T^{-1}\subseteq \mathcal S$ for every invertible 
$T\in M_n(F)$. So, for example, the set of all matrices of rank at most one is the image of some polynomial.  Therefore, Theorem \ref{main} does not hold if the field of complex numbers  is replaced by a finite field. The following question, however, may be asked.

\begin{question} Does Theorem \ref{main} remain valid
if the   field of complex numbers  is replaced by any infinite field $F$? Moreover, if $F$ is not algebraically closed, does Theorem \ref{main} hold for finite-dimensional  simple $F$-algebras $\A$  other than $M_n(F)$? (In this setting, the role of a trace zero matrix is played by a commutator of two elements in $A$.)
\end{question}


\subsection{The Waring problem for matrices of large sizes} If $f$ is a central polynomial for $M_n(F)$, then
its degree  is greater than  $n$. 
 The following question thus remains unanswered.

\begin{question} Let $f$ be a nonconstant polynomial.  Does $f(\A)-f(\A)$ contain all trace zero matrices 
in $\A=M_n(F)$ for all sufficiently large $n$?
\end{question}

A positive answer would be analogous to the solution of the Waring problem for groups \cite{LST}.


\end{document}